\documentclass[10pt]{amsart}
\usepackage{amsmath}
\usepackage{amssymb}
\usepackage{amsthm}
\usepackage[final]{hyperref}
\hypersetup{colorlinks=true, linkcolor=blue, anchorcolor=blue, citecolor=red, filecolor=blue, menucolor=blue, pagecolor=blue, urlcolor=blue}
\usepackage{graphicx}
\newtheorem{thm}{Theorem}
\newtheorem{lem}[thm]{Lemma}
\newtheorem{cor}[thm]{Corollary}

\newtheorem*{PWtheorem}{Paley-Wiener Theorem}
\theoremstyle{definition}

\theoremstyle{definition}
\newtheorem{rem}[thm]{Remark}

\newcommand{\abs}[1]{\vert #1 \vert}

\newcommand{\norm}[1]{\left\Vert #1 \right\Vert}

\title[KdV spatial analyticity]{Lower bounds on the radius of spatial analyticity for the KdV equation}

\author{Sigmund Selberg}
\address{Department of Mathematics\\
University of Bergen\\
PO Box 7803\\
5020 Bergen\\ Norway}
\email{sigmund.selberg@uib.no}

\author{Daniel Oliveira da Silva}
\address{Department of Mathematical Sciences\\
NTNU\\7491 Trondheim\\ Norway
\\
(Current affiliation: Department of Mathematics, Nazarbayev University)}
\email{daniel.dasilva@nu.edu.kz}


\thanks{The authors are supported by the Research Council of Norway, grant no.~213474/F20.}

\begin{document}

\begin{abstract}
We present lower bounds for the uniform radius of spatial analyticity of solutions to the Korteweg-de Vries equation, which improve earlier results due to Bona, Gruji\'c and Kalisch.
\end{abstract}

\maketitle


\section{Introduction}

Consider the Cauchy problem for the Korteweg-de Vries (KdV) equation
\begin{equation}\label{kdv}
\begin{cases}
u_t + u_{xxx} + u u_x = 0, \qquad t,x \in \mathbb R,\\
u(x,0) = u_0(x),
\end{cases}
\end{equation}
where the unknown $u(x,t)$ and the datum $u_0(x)$ are real-valued.

This equation was originally derived by Korteweg and de Vries in \cite{kdv1895} as a model for long waves travelling through a rectangular canal, and has since seen many generalizations.  We are interested in studying well-posedness of \eqref{kdv} for data $u_0$ in a Gevrey-type space $G^{\sigma,s} = G^{\sigma,s}(\mathbb R)$ defined by the norm
\[
\| f \|_{G^{\sigma, s}(\mathbb R)} = \left\| e^{\sigma | \xi |} (1 + | \xi |)^s \hat{f}(\xi) \right\|_{L^2_\xi (\mathbb{R})}.
\]
Here $\hat{f}$ denotes the spatial Fourier transform of $f$. For $\sigma=0$ the space $G^{\sigma,s}$ coincides with the standard Sobolev space $H^s$, and the well-posedness of \eqref{kdv} in these spaces has been studied in many works; see \cite{BS75, ST76, Kato79, Bourgain1993, KPV1993, KPV1996} and the references therein. It is by now known that local well-posedness holds in $H^s$ for $s > -3/4$ (see \cite{KPV1996}) and that for $s \le -3/4$ there is ill-posedness (see \cite{CCT2003, Kishimoto2009, Molinet2011}). For $s \ge 0$ the solutions extend globally in time due to the conservation of $\int u^2(x,t) \, dx$. Our aim here is to extend the well-posedness theory to the spaces $G^{\sigma,s}$. The interest in these spaces is due to the following fact, for which a discussion can be found in \cite[p. 209]{K1976}.
\begin{PWtheorem}
Let $\sigma > 0$, $s \in \mathbb R$.  Then the following are equivalent:
\begin{enumerate}
\item $f \in G^{\sigma, s}$.
\item $f$ is the restriction to the real line of a function $F$ which is holomorphic in the strip
\[
S_\sigma =  \{ x + i y :\ x,y \in \mathbb{R}, | y | < \sigma\}
\]
and satisfies
\[
\sup_{| y | < \sigma} \| F( x + i y ) \|_{H^s_x} < \infty.
\]
\end{enumerate}
\end{PWtheorem}

Thus, a function in $G^{\sigma,s}$ has radius of analyticity at least $\sigma$ at every point on the real line. This fact leads us to consider the following question: given $u_0 \in G^{\sigma,s}$ for some initial radius $\sigma > 0$, how does the radius of analyticity of the solution $u$ evolve in time?

This question has received some attention in the case of the KdV equation and its generalizations.  For short times, it is known that the radius of analyticity remains at least as large as the initial radius; see Gruji\'c and Kalisch \cite{GK2002} for the non-periodic case, and also Li \cite{L2012}, Himonas and Petronilho \cite{HP2012}, and Hannah, Himonas and Petronilho \cite{HHP2011} for the periodic case.  For the global problem, the non-periodic case was studied by Bona, Gruji\'c and Kalisch in \cite{BGK2005}, where it was shown that the radius of analyticity for the KdV equation can decay no faster than $t^{-12}$ as $t \to \infty$ (see Theorem 4 and Corollary 2 in \cite{BGK2005}). In the present paper we improve that result significantly, almost obtaining a rate $t^{-4/3}$. Instead of attempting to obtain a priori estimates on the solution in Gevrey-modified Bourgain spaces directly on any given, large time interval $[0,T]$, as was done in \cite{BGK2005}, we proceed indirectly by decomposing into short subintervals, on each of which we use a short-time local well-posedness result obtained by a contraction argument. On each subinterval we then estimate the growth of the Gevrey-modified version of the conserved quantity $\int u^2(x,t) \, dx$ in terms of the strip width $\sigma > 0$. By taking $\sigma$ sufficiently small we are then able to repeat the local result enough times to reach the target time $T$. This idea was introduced by the second author and Achenef Tesfahun in \cite{ST2015}, where it was applied to the Dirac-Klein-Gordon equations.

The first result we will prove is concerned with short-time persistence of the radius of analyticity. This result extends Theorem 1 from \cite{GK2002}, which covered the range $s \ge 0$.

\begin{thm}\label{local}
Let $\sigma > 0$ and $s > -3/4$. Then for any $u_0 \in G^{\sigma,s}$ there exists a time $\delta = \delta(\| u_0 \|_{G^{\sigma,s}}) > 0$ and a unique solution $u$ of \eqref{kdv} on the time interval $(-\delta,\delta)$ such that
\[
u \in C([-\delta,\delta], G^{\sigma,s}).
\]
Moreover, the solution depends continuously on the data $u_0$, and we have
\[
  \delta = \frac{c_0}{(1+\|u_0\|_{G^{\sigma,s}})^a}
\]
for some constants $c_0 > 0$ and $a > 1$ depending only on $s$.
\end{thm}

Thus for short times the solution remains analytic in the initial strip. Our second and main result yields an estimate on the rate at which the width of the strip can decay with time. Thus $\sigma$ will now depend on time, and we denote its initial value by $\sigma_0$.

\begin{thm}\label{global}
Let $\sigma_0 > 0$ and $s > -3/4$, and assume $u_0 \in G^{\sigma_0,s}$.  The solution $u$ obtained in Theorem \ref{local} extends globally in time, and for any $T > 0$ we have
\[
  u \in C\left([-T,T],G^{\sigma(T),s}\right)
\]
with
\[
\sigma(T) = \min \left\{\sigma_0, c T^{-(4/3 + \varepsilon)} \right\},
\]
where $\varepsilon > 0$ can be taken arbitrarily small and $c > 0$ is a constant depending on $u_0$, $\sigma_0$, $s$ and $\varepsilon$.
\end{thm}

Thus the solution at any time $t$ is analytic in the strip $S_{\sigma(|t|)}$.

Theorem \ref{local} is proved in section \ref{localtheory} by an iteration argument relying on the Gevrey-modified version of a bilinear estimate from \cite{KPV1996}. For the proof of Theorem \ref{global} we must first  establish the existence of an almost conserved quantity guaranteeing that the norm of the solution grows sufficiently slowly so that we can repeatedly apply Theorem \ref{local} enough times to extend the solution up to any time $T > 0$ by taking $\sigma$ small enough.  The derivation of the necessary almost conservation law is contained in section \ref{conlaw}.  The proof of the main result, Theorem \ref{global}, is then given in section \ref{globaltheory}.

At the core of our analysis are some bilinear estimates derived in section \ref{bilinear} as corollaries of an estimate proved in \cite{KPV1996}. We also discuss related counterexamples.

We begin with a discussion of the necessary function spaces in which we will carry out our iteration argument.


\section{Function spaces}

In this section we discuss the function spaces which will be used in the proofs.  First we note the following embedding property of the Gevrey spaces:
\begin{equation}\label{embedding}
  G^{\sigma,s} \subset G^{\sigma',s'} \quad \text{for all $0 < \sigma' < \sigma$ and $s,s' \in \mathbb R$},
\end{equation}
with a corresponding norm inequality $\norm{f}_{G^{\sigma',s'}} \le C_{\sigma,\sigma',s,s'} \norm{f}_{G^{\sigma,s}}$.

 In addition to the spaces $G^{\sigma, s}$, we will also need the Bourgain spaces $X^{s,b}$, defined by the norm
\[
\| u \|_{X^{s,b}} = \left\| (1 + | \xi |)^{s} (1 + | \tau - \xi^3| )^b \widetilde{u} (\xi, \tau) \right\|_{L^2_{\tau, \xi}},
\]
where $\widetilde{u}$ denotes the spacetime Fourier transform,
\[
\widetilde{u}(\xi, \tau) = \int_{\mathbb{R}^2} e^{-i(t \tau + x \xi)} u(x, t)\ dx \, dt.
\]
In addition, we will also need a hybrid of the Gevrey and Bourgain spaces, denoted $X^{\sigma, s, b}$ and defined by the norm
\[
\| u \|_{X^{\sigma, s, b}} = \left\| e^{\sigma | D_x |}u \right\|_{X^{s, b}} = \left\| e^{\sigma | \xi |} (1 + | \xi |)^{s} (1 + | \tau - \xi^3| )^b \widetilde{u} (\xi, \tau) \right\|_{L^2_{\tau, \xi}}.
\]
Here $D_x = -i\partial_x$, which has Fourier symbol $\xi$. Observe that $X^{0,s,b} = X^{s,b}$. Finally, we will need the restrictions of $X^{s, b}$ and $X^{\sigma, s, b}$ to a time slab $\mathbb R \times (-\delta, \delta)$.  These spaces are denoted by $X^{s, b}(\delta)$ and $X^{\sigma, s, b}(\delta)$, respectively, and are Banach spaces when equipped with the norms
\[\begin{aligned}
\| u \|_{X^{s, b}(\delta)} & = \inf \left\{ \| v \|_{X^{s, b}}: \ v = u \text{ on } \mathbb{R} \times (-\delta, \delta) \right\} \\
\| u \|_{X^{\sigma, s, b}(\delta)} & = \inf \left\{ \| v \|_{X^{\sigma, s, b}}: \ v = u \text{ on } \mathbb{R} \times (-\delta, \delta) \right\}.
\end{aligned}\]
By the substitution $u \rightarrow e^{\sigma | D_x |}u$, the properties of $X^{s,b}$ and its restrictions carry over to $X^{\sigma, s, b}$.  The necessary properties are contained in the following lemmas; proofs of the first two lemmas can be found in section 2.6 of \cite{Tao2006}, whereas the third lemma follows by the argument used to prove Lemma 3.1 of \cite{Colliander2004}.

\begin{lem}\label{embed}
Let $\sigma \ge 0$, $s \in \mathbb{R}$ and $b > 1/2$. Then $X^{\sigma,s,b} \subset C(\mathbb R,G^{\sigma,s})$ and
\[
  \sup_{t \in \mathbb{R}} \| u(t) \|_{G^{\sigma, s}} \leq C \| u \|_{X^{\sigma, s, b}},
\]
where the constant $C > 0$ depends only on $b$.
\end{lem}

\begin{lem}\label{exponent}
Let $\sigma \ge 0$, $s \in \mathbb{R}$, $-1/2 < b < b' < 1/2$ and $\delta > 0$.  Then
\[
\| u \|_{X^{\sigma, s, b}(\delta)} \leq C \delta^{b' - b} \| u \|_{X^{\sigma, s, b'}(\delta)},
\]
where $C$ depends only on $b$ and $b'$.
\end{lem}

\begin{lem}\label{restrict}
Let $\sigma \ge 0$, $s \in \mathbb R$, $-1/2 < b < 1/2$ and $\delta > 0$. Then for any time interval $I \subset [-\delta,\delta]$ we have
\[
  \norm{\chi_{I} u}_{X^{\sigma,s,b}} \le C \norm{u}_{X^{\sigma,s,b}(\delta)},
\]
where $\chi_I(t)$ is the characteristic function of $I$, and $C$ depends only on $b$.
\end{lem}

Next, consider the linear Cauchy problem, for given $F(x,t)$ and $u_0(x)$,
\[\begin{cases}
u_t + u_{xxx} = F, \\
u(0) = u_0.
\end{cases}\]
We may write the solution using the Duhamel formula
\[
u(t) = W(t)u_0 + \int_0^t W(t - t') F(t') \ dt',
\]
where $W(t) = e^{- t\partial_{x}^3} = e^{itD_x^3}$ is the solution group; it is the Fourier multiplier with symbol $e^{it\xi^3}$.  Then $u$ satisfies the following $X^{\sigma, s, b}$ energy estimate.
\begin{lem}\label{linear}
Let $\sigma \ge 0$, $s \in \mathbb{R}$, $1/2 < b \leq 1$ and $0 < \delta \leq 1$.  Then for all $u_0 \in G^{\sigma, s}$ and $F \in X^{\sigma, s, b-1}(\delta)$, we have the estimates
\[\begin{aligned}
& \qquad \qquad \| W(t) u_0 \|_{X^{\sigma, s, b}(\delta)} \leq C \| u_0 \|_{G^{\sigma, s}},
\\
& \left\| \int_0^t W(t - t') F(t')\ dt' \right\|_{X^{\sigma, s, b}(\delta)} \leq C \| F \|_{X^{\sigma, s, b - 1}(\delta)},
\end{aligned}\]
where the constant $C > 0$ depends only on $b$.
\end{lem}


\section{Bilinear estimates}\label{bilinear}

In this section we derive the bilinear estimates that lie at the core of the proofs of Theorems \ref{local} and \ref{global}. The estimates will be obtained as corollaries of the following key estimate, from \cite{KPV1996}. At the end of this section we then discuss some related counterexamples and their implications.

\begin{thm}[Kenig, Ponce and Vega {\cite[Thm.~2.2]{KPV1996}}]\label{KPVthm}
Given $s > -3/4$, there exist $b \in (1/2,1)$ and $\varepsilon > 0$ such that such that the following estimate holds for any $b' \in [b,b+\varepsilon)$:
\[
  \| \partial_x (uv) \|_{X^{s, b' - 1}} \leq C \| u \|_{X^{s, b}} \| v \|_{X^{s, b}}.
\]
Here $C > 0$ is a constant depending only on $s$, $b$ and $b'$. 
\end{thm}

\begin{rem}\label{KPVrem1}
The actual ranges of $b$ and $b'$ can be found in \cite{KPV1996}. For example, if $s=0$ (the most important case for us), then $1/2 < b \le b' \le 3/4$ suffices.
\end{rem}

\begin{rem}\label{KPVrem2}
By Plancherel's theorem, the estimate in Theorem \ref{KPVthm} can be restated as follows:
\begin{multline}\label{KeyBilinear}
\Biggl\| \frac{\xi (1+|\xi|)^s}{(1 + | \tau - \xi^3 |)^{1 - b'}} \int_{\mathbb{R}^2} \frac{f(\xi - \xi_1, \tau - \tau_1)}{(1+|\xi-\xi_1|)^s(1 + | \tau - \tau_1 - (\xi - \xi_1)^3|)^b} \times
\\
\times \frac{g( \xi_1, \tau_1)}{(1+|\xi_1|)^s(1 + | \tau_1 - \xi_1^3|)^b} \ d \xi_1 d \tau_1 \Biggr\|_{L^2_{\xi, \tau}} \leq C \left\| f \right\|_{L^2_{\xi, \tau}} \left\| g \right\|_{L^2_{\xi, \tau}}.
\end{multline}
Here $f(\xi, \tau) = (1+|\xi|)^s (1 + | \tau - \xi^3| )^b \widetilde{u}(\xi, \tau)$ and ditto for $g$ and $\widetilde v$, so that $\| u \|_{X^{s, b}} = \left\| f \right\|_{L^2_{\xi, \tau}}$ and $\| v \|_{X^{s, b}} = \left\| g \right\|_{L^2_{\xi, \tau}}$. 
\end{rem}

The first corollary provides the key to proving Theorem \ref{local}.

\begin{cor}\label{Cor1}
Given $s > -3/4$, there exist $b \in (1/2,1)$, $b' \in (b,1)$ and $C > 0$ such that the following estimate holds for all $\sigma \ge 0$:
\[
  \| \partial_x (uv) \|_{X^{\sigma, s, b' - 1}} \leq C \| u \|_{X^{\sigma, s, b}} \| v \|_{X^{\sigma, s, b}}.
\]
\end{cor}

\begin{proof}
This estimate can be restated in the form \eqref{KeyBilinear} modified by the factor
\[
  \frac{e^{\sigma|\xi|}}{e^{\sigma|\xi-\xi_1|}e^{\sigma|\xi_1|}}
\]
inserted into the integral on the left side. But this factor is $\le 1$ by the triangle inequality $|\xi| \le |\xi-\xi_1| + |\xi_1|$, hence the desired estimate reduces to \eqref{KeyBilinear}, that is, to Theorem \ref{KPVthm}.
\end{proof}

To state the second corollary, we introduce the bilinear operator $B_\rho$, for $\rho \ge 0$, defined on the Fourier transform side by
\[
  \widetilde{B_\rho(u,v)}(\xi,\tau) = \int_{\mathbb R^2} \left[\min(|\xi-\xi_1|,|\xi_1|)\right]^\rho
  \widetilde{u}(\xi-\xi_1, \tau-\tau_1) \, \widetilde{v}(\xi_1, \tau_1) \ d \xi_1 d \tau_1.
\]
For this operator we have the following estimate, which is crucial to the proof of Theorem \ref{global}.

\begin{cor}\label{Cor2}
Given $\rho \in [0,3/4)$, there exist $b \in (1/2,1)$ and $C > 0$ such that
\[
  \| \partial_x B_\rho(u,v) \|_{X^{0, b - 1}} \leq C \| u \|_{X^{0, b}} \| v \|_{X^{0, b}}.
\] 
\end{cor}

\begin{proof} The desired estimate can be restated as
\begin{multline}\label{Cor2a}
\Biggl\| \frac{\xi}{(1 + | \tau - \xi^3 |)^{1 - b}} \int_{\mathbb{R}^2}
[\min(|\xi-\xi_1|,|\xi_1|)]^\rho
\frac{f(\xi - \xi_1, \tau - \tau_1)}{(1 + | \tau - \tau_1 - (\xi - \xi_1)^3|)^b} \times
\\
\times \frac{g( \xi_1, \tau_1)}{(1 + | \tau_1 - \xi_1^3|)^b} \ d \xi_1 d \tau_1 \Biggr\|_{L^2_{\xi, \tau}} \leq C \left\| f \right\|_{L^2_{\xi, \tau}} \left\| g \right\|_{L^2_{\xi, \tau}}.
\end{multline}
But from the triangle inequality it is seen that
\[
  \min(|\xi-\xi_1|,|\xi_1|)
  \le
  2\frac{(1+|\xi-\xi_1|)(1+|\xi_1|)}{(1+|\xi|)},
\]
and taking this to the power $\rho \ge 0$ we conclude that the left side of \eqref{Cor2a} is bounded by $2^\rho$ times the left side of \eqref{KeyBilinear} with $s=-\rho > -3/4$ and $b=b'$. Thus the desired estimate reduces to Theorem \ref{KPVthm}.
\end{proof}

The last corollary will be used in combination with the following estimate.
\begin{lem}\label{symbol}
For $\sigma > 0$, $\theta \in [0,1]$ and $\alpha,\beta \in \mathbb R$, we have the estimate
\begin{equation}\label{symbolest}
e^{\sigma | \alpha |} e^{\sigma | \beta |} - e^{\sigma | \alpha + \beta |} \leq \left[2 \sigma \min(| \alpha |, | \beta |)\right]^{\theta} e^{\sigma | \alpha |}e^{\sigma | \beta |}.
\end{equation}
\end{lem}
\begin{proof}
If $\alpha$ and $\beta$ have the same sign, then the left side of \eqref{symbolest} equals zero and the inequality holds trivially, therefore we assume that $\alpha$ and $\beta$ have opposite signs.  Without loss of generality, assume $\alpha \geq 0$ and $\beta \leq 0$.  If $| \beta | \leq | \alpha |$, then $\alpha + \beta \geq 0$, and so the left side of \eqref{symbolest} becomes
\[\begin{aligned}
e^{\sigma \alpha}e^{- \sigma \beta} - e^{\sigma (\alpha + \beta)} & = e^{\sigma(\alpha + \beta)}\left( e^{-2\sigma \beta} - 1 \right) \\
& \leq (2 \sigma | \beta |)^{\theta} e^{- 2 \sigma \beta} e^{\sigma (\alpha + \beta) } \\
& = (2 \sigma | \beta |)^{\theta} e^{\sigma | \alpha |} e^{\sigma | \beta |}.
\end{aligned}\]
Here we have used the fact that, for $x \geq 0$, the inequalities $e^x - 1 \leq e^x$ and $e^x - 1 \leq xe^x$ both hold, hence also
\[
  e^x-1 \le x^\theta e^x \qquad \text{for $x \ge 0$ and $\theta \in [0,1]$}.
\]
On the other hand, if $| \beta | \geq | \alpha |$, then $\alpha + \beta \leq 0$, so the left side of \eqref{symbolest} becomes
\[\begin{aligned}
e^{\sigma \alpha}e^{- \sigma \beta} - e^{-\sigma (\alpha + \beta)} & = e^{- \sigma(\alpha + \beta)}\left( e^{2\sigma \alpha} - 1 \right) \\
& \leq (2 \sigma | \alpha |)^{\theta} e^{2 \sigma \alpha} e^{- \sigma (\alpha + \beta) } \\
& = (2 \sigma | \alpha |)^{\theta} e^{\sigma | \alpha |} e^{\sigma | \beta |}.
\end{aligned}\]
The result follows.
\end{proof}

Finally, we prove a counterexample which shows that it is essential that there is a \emph{minimum} in the symbol of the operator $B_\rho$ appearing in Corollary \ref{Cor2}. That is, if the minimum is replaced by $|\xi-\xi_1|^\rho$ or $|\xi_1|^\rho$, or even $|\xi|^\rho$, then the corresponding estimate fails for every $\rho > 0$.

\begin{thm}\label{Counter1}
Let $\sigma,s,b,b' \in \mathbb R$ and $\rho > 0$. Then the following estimates fail:
\begin{align}
  \label{False1}
  \norm{\partial_x\left(u \cdot \abs{D_x}^\rho v \right)}_{X^{\sigma,s,b'-1}} &\le C \norm{u}_{X^{\sigma,s,b}} \norm{v}_{X^{\sigma,s,b}},
  \\
  \label{False2}
  \norm{\partial_x \abs{D_x}^\rho \left(u v \right)}_{X^{\sigma,s,b'-1}} &\le C \norm{u}_{X^{\sigma,s,b}} \norm{v}_{X^{\sigma,s,b}}.
\end{align}
\end{thm}

\begin{proof} The failure happens in a low-high frequency interaction where the left sides of \eqref{False1} and \eqref{False2} are comparable, hence we only need to disprove \eqref{False2}.

By $L^2$ duality it suffices to disprove
\[
  \abs{I}
  \le C
  \norm{f}_{L^2}\norm{g}_{L^2}\norm{h}_{L^2},
\]
where
\[
  I = \int_{\mathbb R^4} \xi\abs{\xi}^{\rho} \kappa_1 \kappa_2 \,
  f(\tau_1,\xi_1) g(\tau-\tau_1,\xi-\xi_1) h(\tau,\xi) \, d\tau_1 \, d\xi_1 \, d\tau \, d\xi,
\]
\[
  \kappa_1
  =
  \frac{e^{\sigma\abs{\xi}}(1+\abs{\xi})^s}
  {e^{\sigma\abs{\xi-\xi_1}}(1+\abs{\xi-\xi_1})^s e^{\sigma\abs{\xi_1}}(1+\abs{\xi_1})^s},
\]
\[
  \kappa_2
  =
  \frac{1}{(1+\abs{\tau-\xi^3})^{1-b'}(1+\abs{\tau_1-\xi_1^3})^{b}(1+\abs{\tau-\tau_1-(\xi-\xi_1)^3})^{b}}.
\]
Let $N \gg 1$ be a parameter to be sent to infinity. Let $f(\tau_1,\xi_1)$ be the characteristic function of the region
\begin{equation}\label{fRegion}
  \frac{1}{N^2} \le \xi_1 \le \frac{2}{N^2}, \qquad \abs{\tau_1-\xi_1^3} \le 1,
\end{equation}
and let $h(\tau,\xi)$ be the characteristic function of
\begin{equation}\label{hRegion}
  N \le \xi \le 2N, \qquad \abs{\tau-\xi^3} \le 1.
\end{equation}
Then for $N$ large,
\begin{gather*}
  \frac{N}{2} \le \xi-\xi_1 \le 2N,
  \\
  \abs{\tau-\tau_1-(\xi-\xi_1)^3} = \abs{\tau-\xi^3 - (\tau_1-\xi_1^3) + 3\xi\xi_1(\xi-\xi_1)} \le 50,
\end{gather*}
and we let $g(\tau-\tau_1,\xi-\xi_1)$ be the corresponding characteristic function.

Now one observes that
\[
  \kappa_1, \kappa_2 \sim 1 \quad \text{and} \quad I \sim N^{1+\rho} N \frac{1}{N^2} = N^\rho,
\]
while
\[
  \norm{f}_{L^2}\norm{g}_{L^2}\norm{h}_{L^2} \sim \left( \frac{1}{N^2} N N \right)^{1/2} = 1.
\]
Letting $N$ tend to infinity we then get the counterexample for any $\rho > 0$.
\end{proof}


\section{Proof of Theorem \ref{local}}\label{localtheory}

Fix $\sigma > 0$, $s > -3/4$ and  $u_0 \in G^{\sigma,s}$.  In order to construct the local solution $u$ to \eqref{kdv} we proceed by an iteration argument in the space $X^{\sigma, s, b}(\delta)$.  Let $\{u^{(n)}\}_{n = 0}^{\infty}$ be the sequence defined by
\[
\begin{cases}
u_{t}^{(0)} + u_{xxx}^{(0)} = 0, \\
u^{(0)}(0) = u_0,
\end{cases} \qquad
\begin{cases}
u_{t}^{(n)} + u_{xxx}^{(n)} = - \frac{1}{2} \partial_x \left( u^{(n - 1)} u^{(n - 1)} \right), \\
u^{(n)}(0) = u_0,
\end{cases}
\]
for $n \in \{1,2,\dots\}$. Based on the comments preceding Lemma \ref{linear}, we may write
\[\begin{aligned}
& \qquad \qquad \qquad \qquad u^{(0)}(x,t) = W(t) u_0(x), \\
& u^{(n)}(x,t) = W(t)u_0(x) - \frac{1}{2} \int_0^t W(t - t') \partial_x \left(u^{(n - 1)}(x,t') u^{(n - 1)}(x,t') \right) \ dt'.
\end{aligned}\]
It then follows from Lemmas \ref{exponent} and \ref{linear} that
\begin{align}
\label{energy1}
& \qquad \qquad \qquad \qquad \qquad \| u^{(0)} \|_{X^{\sigma, s, b}(\delta)} \leq C \| u_0 \|_{G^{\sigma,s}},
\\
\label{energy2}
& \| u^{(n)} \|_{X^{\sigma, s, b}(\delta)} \leq C \| u_0 \|_{G^{\sigma,s}} + C \delta^{b' - b} \left\| \partial_x \left(u^{(n - 1)} u^{(n - 1)} \right) \right\|_{X^{\sigma, s, b' - 1}(\delta)}.
\end{align}
Choosing $1/2 < b < b' < 1$ as in Corollary \ref{Cor1}, and applying the estimate from the Corollary (restricted to a time-slab) to \eqref{energy2}, we obtain
\[
\| u^{(n)} \|_{X^{\sigma, s, b}(\delta)} \leq C \| u_0 \|_{G^{\sigma,s}} + C \delta^{b' - b} \| u^{(n-1)} \|^2_{X^{\sigma, s, b}(\delta)}.
\]
By induction it follows that
\begin{equation}\label{KeyBound}
\| u^{(n)} \|_{X^{\sigma, s, b}(\delta)} \leq 2 C \| u_0 \|_{G^{\sigma,s}}
\end{equation}
for all $n$, if $\delta \in (0,1]$ is chosen so small that
\[
\delta \leq \frac{1}{\left( 8 C^2 \| u_0 \|_{G^{\sigma,s}}\right)^{\frac{1}{b' - b}}}.
\]
With this choice, one has moreover, applying Corollary \ref{Cor1} and the energy estimate once again, and making use of the bound \eqref{KeyBound},
\[\begin{aligned}
&\| u^{(n)} - u^{(n-1)} \|_{X^{\sigma, s, b}(\delta)} \\
& \qquad\leq C \delta^{b' - b}(\| u^{(n-1)} \|_{X^{\sigma, 0, b}(\delta)} + \| u^{(n-2)} \|_{X^{\sigma, s, b}(\delta)}) \| u^{(n - 1)} - u^{(n - 2)} \|_{X^{\sigma, s, b}(\delta)} \\
& \qquad \leq \frac{1}{2}\| u^{(n - 1)} - u^{(n - 2)} \|_{X^{\sigma, s, b}(\delta)}.
\end{aligned}\]
It follows that the sequence converges to a solution $u$ verifying the bound \eqref{KeyBound}.

For the continuous dependence on the initial data, assume $u$ and $v$ are solutions to \eqref{kdv} for data $u_0$ and $v_0$, respectively.  By an argument similar to the one above, it transpires that for any $\delta' \in (0,\delta)$, with $\delta$ as above, one has the inequality
\[
\| u - v \|_{X^{\sigma, s, b}(\delta')} \leq C\| u_0 - v_0 \|_{G^{\sigma,s}} + \frac{1}{2}\| u - v \|_{X^{\sigma, s, b}(\delta')}
\]
provided $\| u_0 - v_0 \|_{G^{\sigma,s}}$ is sufficiently small. This proves continuous dependence.

Finally, we prove the (unconditional) uniqueness of solutions.  Suppose $u,v \in C_tG^{\sigma,s}$ are both solutions corresponding to initial data $u_0$, and let $w = u - v$.  Then $w$ obeys the equation $w_t + w_{xxx} + w u_x + v w_x = 0$.  Multiplying by $w$ and integrating in $x$ gives us the inequality (using $2v w_x w = (vw^2)_x - v_x w^2$)
\[
\frac{d}{dt} \| w(t) \|_{L_x^2}^2 \leq \left(\| u_x(t) \|_{L_x^{\infty}} + \| v_x(t) \|_{L_x^{\infty}}\right) \| w(t) \|_{L_x^{2}}^2.
\]
Since $u,v \in C_tG^{\sigma,s}$, it follows that each of the $L^{\infty}$ norms is bounded. It then follows from Gr\"onwall's inequality that $w = 0$. 


\section{Approximate Conservation Law}\label{conlaw}

Now that we have established the existence of local solutions, we would like to apply the local result repeatedly to cover time intervals of arbitrary length.  This of course requires some sort of control on the growth of the  norm on which the local existence time depends.  This control is afforded by the following approximate conservation law, which in the limit $\sigma \to 0$ reduces to the familiar conservation of $\int u^2(x,t) \, dx$.  The approximate conservation law will allow us (see the next section) to repeat the local result on successive short time intervals to reach any target time $T > 0$, by adjusting the strip width parameter $\sigma$ according to the size of $T$.

\begin{thm}\label{approx}
Given $\rho \in [0,3/4)$, there exist $b \in (1/2,1)$ and $C > 0$ such that for any $\delta,\sigma > 0$ and any solution $u \in X^{\sigma, 0, b}(\delta)$ to the Cauchy problem \eqref{kdv} on the time interval  $(-\delta,\delta)$, we have the estimate
\[
\sup_{\abs{t} \le \delta} \| u(t) \|^2_{G^{\sigma,0}} \leq \| u(0) \|^2_{G^{\sigma,0}} + C \sigma^{\rho} \| u \|^3_{X^{\sigma, 0, b}(\delta)}.
\]
\end{thm}

\begin{rem}\label{keyremark}
Applying this estimate to the local solution $u$ from Theorem \ref{local}, and recalling that $u$ verifies the bound \eqref{KeyBound}, one obtains
\[
  \sup_{\abs{t} \le \delta} \| u(t) \|^2_{G^{\sigma,0}} \leq \| u(0) \|^2_{G^{\sigma,0}} + C \sigma^{\rho} \| u(0) \|^3_{G^{\sigma,0}},
\]
with $\delta$ as in Theorem \ref{local}.
\end{rem}

For the proof of Theorem \ref{approx}, we require the following preliminary estimate.

\begin{lem}\label{prodlem}
Let $F$ be given by
\[
F = \frac{1}{2} \partial_x \left(e^{\sigma | D_x |}(u \cdot u) - e^{\sigma | D_x |}u \cdot e^{\sigma | D_x |}u \right).
\]
Given $\rho \in [0,3/4)$, there exist $b \in (1/2,1)$ and $C > 0$ such that for all $\sigma > 0$ and $u \in X^{\sigma,0,b}$ we have
\begin{equation}\label{prodest}
\| F \|_{X^{0,b-1}} \leq C \sigma^{\rho} \left\| u \right\|_{X^{\sigma, 0, b}}^2.
\end{equation}
\end{lem}

\begin{proof} By Lemma \ref{symbol} and Corollary \ref{Cor2} we have
\[
  \| F \|_{X^{0,b-1}}
  \le \frac12 (2\sigma)^\rho \| \partial_x B_\rho(e^{\sigma\abs{D_x}}u,e^{\sigma\abs{D_x}}u) \|_{X^{0,b-1}}
  \le C \sigma^\rho \left\| u \right\|_{X^{\sigma, 0, b}}^2.
\]
\end{proof}

\begin{proof}[Proof of Theorem \ref{approx}]
Let $U(x,t) = e^{\sigma | D_x | } u(x,t)$, which is real-valued since $u$ is and since the multiplier is even.  By applying the operator $e^{\sigma | D_x |}$ to \eqref{kdv}, we can see that $U$ satisfies the equation
\begin{equation} \label{modkdv}
U_t + U_{xxx} + U U_{x} = F,
\end{equation}
where
\[
F = e^{\sigma | D_x |}u \cdot e^{\sigma | D_x |}u_x - e^{\sigma | D_x |}(u u_x) = \frac{1}{2} \partial_x \left( e^{\sigma | D_x |}u \cdot e^{\sigma | D_x |}u - e^{\sigma | D_x |}(u \cdot u) \right).
\]
We multiply both sides of \eqref{modkdv} by $U$ and integrate in space, obtaining
\begin{equation}\label{integrated}
\int_{\mathbb{R}} U U_t \ dx + \int_{\mathbb{R}} U U_{xxx} \ dx + \int_{\mathbb{R}} U^2 U_x  \ dx = \int_{\mathbb{R}} U F \ dx.
\end{equation}
Integration by parts is justified, since we may assume that $U(t,x)$ decays to zero as $| x | \to \infty$, and the same holds for all spatial derivatives.\footnote{Indeed, we are aiming to prove \eqref{prodest} for a given $\sigma > 0$, but by the monotone convergence theorem it suffices to prove it for all $\sigma' < \sigma$ (the constant $C$ being uniform). For $U' = e^{\sigma' | D_x | } u$ we get by Cauchy-Schwarz and by the assumption that $u \in X^{\sigma,0,b} \subset L_t^\infty G^{\sigma,0}$,\[
  \int | \widehat{\partial_x^j U'}(t,\xi) | \, d\xi
  =
  \int | e^{(\sigma'-\sigma) | \xi | } e^{\sigma | \xi | } \xi^j \widehat u(t,\xi) | \, d\xi
  \le
  \left( \int \xi^{2j} e^{-2\varepsilon | \xi | } \, d\xi \right)^{1/2} \| u(t) \|_{G^{\sigma,0}} < \infty,
\]
where $\varepsilon = \sigma-\sigma' > 0$ and $j \in \{0,1,\dots\}$. So by Riemann-Lebesgue, $\partial_x^j U'(t,x) \to 0$ as $| x | \to \infty$.}
Thus, \eqref{integrated} can be rewritten as
\[
\frac{1}{2} \frac{d}{dt}\int_{\mathbb{R}} U^2 \ dx - \frac{1}{2} \int_{\mathbb{R}} \partial_x( U_x U_x ) \ dx + \frac{1}{3} \int_{\mathbb{R}} \partial_x (U^3) \ dx = \int_{\mathbb{R}} U F \ dx,
\]
and moreover, the second and third terms on the left side vanish. Integration in time then yields
\[
\| u(\delta) \|_{G^{\sigma,0}}^2 \leq \| u(0) \|_{G^{\sigma,0}}^2 + 2 \left| \int_{\mathbb{R}^2} \chi_{[0,\delta]}(t) \cdot U F  \ dx dt \right|.
\]
By applying Parseval's identity and H\"older's inequality, we can estimate the integral on the right side by
\begin{align*}
\left| \int_{\mathbb{R}^2} \chi_{[0,\delta]}(t) \cdot U F \ dx dt \right|
&\leq \| \chi_{[0,\delta]}(t) U \|_{X^{0,1 - b}} \| \chi_{[0,\delta]}(t) F \|_{X^{0,b - 1}}
\\
&\le C \| U \|_{X^{0,1 - b}(\delta)} \| F \|_{X^{0,b - 1}(\delta)},
\end{align*}
where we used Lemma \ref{restrict} to get the last inequality. Using now the estimate obtained in Lemma \ref{prodlem} (restricted to a time-slab) and the fact that $1 - b < b$ since $b > 1/2$, we obtain
\[
\| u(\delta) \|^2_{G^{\sigma,0}} \leq \| u(0) \|^2_{G^{\sigma,0}} + C \sigma^{\rho} \| u \|^3_{X^{\sigma, 0, b}(\delta)},
\]
and the theorem is proved.
\end{proof}


\section{Proof of Theorem \ref{global}}\label{globaltheory}

Fix $\sigma_0 > 0$, $s > -3/4$ and $u_0 \in G^{\sigma_0,s}$. Moreover, fix $\rho \in (0,3/4)$ and let $b = b(\rho) \in (1/2,1)$ be as in Theorem \ref{approx}. By invariance of the KdV equation under the reflection $(t,x) \to (-t,-x)$, we may restrict to positive times. Thus, it suffices to prove that the solution $u$ to \eqref{kdv} satisfies
\[
  u \in C\left([0,T],G^{\sigma(T),s}\right) \quad \text{for all $T > 0$},
\]
where
\[
\sigma(T) = \min \left\{\sigma_0, c T^{-1/\rho} \right\}
\]
and $c > 0$ is a constant depending on $u_0$, $\sigma_0$, $s$ and $\rho$.

By Theorem \ref{local}, there is a maximal time $T^* = T^*(u_0,\sigma_0,s) \in (0,\infty]$ such that
\[
  u \in C([0,T^*);G^{\sigma_0,s}).
\]
If $T^*=\infty$, we are done. If $T^* < \infty$, as we assume henceforth, it remains to prove
\begin{equation}\label{aim}
  u \in C\left([0,T];G^{cT^{-1/\rho},s}\right) \quad \text{for all $T \ge T^*(u_0,\sigma_0,s)$}.
\end{equation}
We first prove this in the case $s=0$. Then at the end of this section we do the general case, which essentially reduces to $s=0$.

\subsection{The case $s=0$} Fix $T \ge T^*$. Defining
\[
  M_\sigma(t) = \| u(t) \|_{G^{\sigma,0}},
\]
we will show that, for $\sigma > 0$ sufficiently small,
\begin{equation}\label{keybound}
  M^2_\sigma(t) \le 2M^2_{\sigma_0}(0) \qquad \text{for $t \in [0,T]$}.
\end{equation}
To prove this we will use repeatedly Theorems \ref{local} and \ref{approx} with the time step
\begin{equation}\label{delta}
  \delta = \frac{c_0}{[1+2M_{\sigma_0}(0)]^a},
\end{equation}
where $c_0 > 0$ and $a > 1$ are as in Theorem \ref{local} (with $s=0$). The smallness conditions on $\sigma$ will be
\begin{equation}\label{sigma0}
  \sigma \le \sigma_0
\end{equation}
and
\begin{equation}\label{sigma}
  \frac{2T}{\delta} C\sigma^\rho 2^{3/2} M_{\sigma_0}(0) \le 1,
\end{equation}
where $C > 0$ is the constant in Remark \ref{keyremark}. As we will see, \eqref{sigma} actually implies \eqref{sigma0}, but for the moment we keep both conditions.

Proceeding by induction we will verify that
\begin{align}
  \label{induction1}
  \sup_{t \in [0,k\delta]} M^2_\sigma(t) &\le M^2_\sigma(0) + k C\sigma^\rho 2^{3/2} M^3_{\sigma_0}(0),
  \\
  \label{induction2}
  \sup_{t \in [0,k\delta]} M^2_\sigma(t) &\le 2M^2_{\sigma_0}(0),
\end{align}
for $k \in \{1,\dots,n+1\}$, where $n \in \mathbb N$ is chosen so that $T \in [n\delta,(n+1)\delta)$. This $n$ does exist, since by Theorem \ref{local} and the definition of $T^*$ we have $\delta < \frac{c_0}{[1+M_{\sigma_0}(0)]^a} < T^*$, hence $\delta < T$.

In the first step we cover the interval $[0,\delta]$, and by Theorem \ref{approx} and the remark following it we have
\[
  \sup_{t \in [0,\delta]} M^2_\sigma(t) \le M^2_\sigma(0) + C\sigma^\rho M^3_\sigma(0)
  \le M^2_\sigma(0) + C\sigma^\rho M^3_{\sigma_0}(0),
\]
where we used that $M_\sigma(0) \le M_{\sigma_0}(0)$, since $\sigma \le \sigma_0$. This verifies \eqref{induction1} for $k=1$, and now \eqref{induction2} follows by using again $M_\sigma(0) \le M_{\sigma_0}(0)$ as well as $C\sigma^\rho M_{\sigma_0}(0) \le 1$. The latter follows from \eqref{sigma}, since $\delta < T$.

Next, assuming that \eqref{induction1} and \eqref{induction2} hold for some $k \in \{1,\dots,n\}$, we will prove that they hold for $k+1$. We estimate
\[\begin{alignedat}{2}
  \sup_{t \in [k\delta,(k+1)\delta]} M^2_\sigma(t) &\le M^2_\sigma(k\delta) + C\sigma^\rho M^3_\sigma(k\delta)& \qquad &\text{by Thm.~\ref{approx}}
  \\
  &\le M^2_\sigma(k\delta) + C\sigma^\rho 2^{3/2} M^3_{\sigma_0}(0)& \qquad &\text{by \eqref{induction2}}
  \\
  &\le
  M^2_\sigma(0) + k C\sigma^\rho 2^{3/2} M^3_{\sigma_0}(0)
  + C\sigma^\rho 2^{3/2} M^3_{\sigma_0}(0)& \qquad &\text{by \eqref{induction1}},
\end{alignedat}\]
verifying \eqref{induction1} with $k$ replaced by $k+1$. To get \eqref{induction2} with $k$ replaced by $k+1$, it is then enough to have
\[
  (k+1) C\sigma^\rho 2^{3/2} M^3_{\sigma_0}(0) \le M^2_{\sigma_0}(0),
\]
but this holds by \eqref{sigma}, since $k+1 \le n+1 \le T/\delta+1 < 2T/\delta$.

We have thus proved \eqref{keybound} under the smallness assumptions \eqref{sigma0} and \eqref{sigma} on $\sigma$. Since $T \ge T^*$, the condition \eqref{sigma} must fail for $\sigma=\sigma_0$, that is, the left side must be strictly larger than $1$, since otherwise we would be able to continue the solution in $G^{\sigma_0,0}$ beyond the time $T$, contradicting the maximality of $T^*$. Therefore, there must be some $\sigma \in (0,\sigma_0)$ for which equality holds in \eqref{sigma}, and using \eqref{delta} we get
\[
  2T \frac{[1+2M_{\sigma_0}(0)]^a}{c_0} C\sigma^\rho 2^{3/2} M_{\sigma_0}(0) = 1
\]
hence
\[
  \sigma = c T^{-1/\rho}, \quad \text{where} \quad
  c = \left( \frac{c_0}{C 2^{5/2} M_{\sigma_0}(0) [1+2M_{\sigma_0}(0)]^a }\right)^{1/\rho}.
\]
We have proved that \eqref{keybound} holds for this $\sigma$, hence $M_\sigma(t) < \infty$ for $t \in [0,T]$, and this completes the proof of \eqref{aim} for the case $s=0$.

\subsection{The general case}

For general $s$ we use the embedding \eqref{embedding} to get
\[
  u_0 \in G^{\sigma_0,s} \subset G^{\sigma_0/2,0}.
\]
The case $s=0$ already being proved, we know that there is a $T_0 > 0$ such that
\[
  u \in C\left([0,T_0),G^{\sigma_0/2,0}\right)
\]
and
\[
  u \in C\left([0,T],G^{2 \kappa T^{-1/\rho},0}\right) \quad \text{for $T \ge T_0$},
\]
where $\kappa > 0$ depends on $u_0$, $\sigma_0$ and $\rho$. Applying again the embedding \eqref{embedding} we now conclude that
\[
  u \in C\left([0,T_0),G^{\sigma_0/4,s}\right)
\]
and
\[
  u \in C\left([0,T],G^{\kappa T^{-1/\rho},s}\right) \quad \text{for $T \ge T_0$},
\]
and these together imply \eqref{aim}, completing the proof of Theorem \ref{global}.

\bibliographystyle{amsplain}
\bibliography{database}

\end{document}